\documentclass{birkjour}
 \newtheorem{thm}{Theorem}[section]
 
 \newtheorem{lem}[thm]{Lemma}
 
 \theoremstyle{definition}
 \newtheorem{defn}[thm]{Definition}
 \theoremstyle{remark}

 \usepackage[mathscr]{eucal}
 \usepackage[OT1,T1,TS1,T2A]{fontenc}
\usepackage[cp1251]{inputenc}
\usepackage[russian,english]{babel}
 \numberwithin{equation}{section}

\DeclareMathOperator{\tr}{trace}

\begin{document}

%
%
%
%
%

\title[On a special case of the Herbert Stahl theorem]
 {On a special case of the Herbert Stahl theorem}
\author[V. Katsnelson]{Victor Katsnelson}
\address{%
Department of Mathematics\\
The Weizmann Institute\\
76100, Rehovot\\
Israel}
\email{victor.katsnelson@weizmann.ac.il; victorkatsnelson@gmail.com}
\date{July 1, 2016}
\begin{abstract}
The BMV conjecture states that for \(n\times n\) Hermitian matrices \(A\) and \(B\)
the function \(f_{A,B}(t)=\tr e^{tA+B}\) is exponentially convex.
Recently the BMV conjecture was proved by Herbert Stahl. The proof of Herbert Stahl is based
on ingenious considerations related to Riemann surfaces of algebraic functions.
In the present paper we give a purely "matrix" proof of the BMV conjecture for
the special case \(\textup{rank}\,A=1\). This proof is based on the Lie product formula
for the exponential of the sum of two matrices and
does not require complex analysis.
\end{abstract}
\maketitle
 \section{Herbert Stahl's Theorem.}
In the paper \cite{BMV} a conjecture was
formulated which now is commonly known as the BMV conjecture:\\[1.0ex]
\textbf{The BMV Conjecture.} Let \(A\) and \(B\) be Hermitian matrices of size
\(n\times{}n\).
Then the function
\begin{equation}
\label{TrF}
f_{A,B}(t)=
\textup{trace}\,\{\exp[tA+B]\}
\end{equation}
of the variable \(t\) is representable as a bilateral Laplace transform of a \textsf{non-negative} measure
\(d\sigma_{A,B}(\lambda)\) compactly supported on the real axis:
\begin{equation}
\label{LaR}
f_{A,B}(t)=\!\!\int\limits_{\lambda\in(-\infty,\infty)}\!\!\exp(t\lambda)\,d\sigma_{A,B}(\lambda), \ \ \forall \,t\in(-\infty,\infty).
\end{equation}

Let us note that the function \(f_{A,B}(t)\), considered for \(t\in\mathbb{C}\), is an entire function of exponential type. The indicator diagram of the function \(f_{A,B}\) is
the closed interval \([\lambda_{\min},\lambda_{\max}]\), where \(\lambda_{\min}\) and
\(\lambda_{\max}\) are the least and the greatest eigenvalues of the matrix \(A\) respectively.
Thus if the function \(f_{A,B}(t)\) is representable in the form \eqref{LaR} with a non-negative
measure \(d\sigma_{A,B}(\lambda)\), then \(d\sigma_{A,B}(\lambda)\) is actually supported on the interval
\([\lambda_{\min},\lambda_{\max}]\) and the representation
\begin{equation}
\label{LaRm}
f_{A,B}(t)=
\hspace*{-8.0pt}\int\limits_{\lambda\in[\lambda_{\min},\lambda_{\max}]}\hspace*{-10.0pt}
\exp(t\lambda)\,\,d\sigma_{A,B}(\lambda), \ \ \forall \,t\in\mathbb{C},
\end{equation}
holds for every \(t\in\mathbb{C}\).

The representability of the function \(f_{A,B}(t)\), \eqref{TrF}, in the form \eqref{LaRm}
with a non-negative \(d\sigma_{A,B}\) is evident if the matrices \(A\) and \(B\) commute.
In this case \(d\sigma(\lambda)\) is an atomic measure supported on the spectrum of the matrix \(A\).
In general case, if the matrices \(A\) and \(B\) do not commute, the BMV conjecture remained an
open question for longer than 35 years. In 2011, Herbert
Stahl proved the BMV conjecture.\\[1.0ex]
\textbf{Theorem}\,(H.Stahl) \textit{Let \(A\) and \(B\) be \(n\times{}n\) hermitian matrices.}
\textit{Then the function \(f_{A,B}(t)\) defined by \eqref{TrF} is representable
as the bilateral Laplace transform \eqref{LaRm} of a non-negative measure \(d\sigma_{A,B}(\lambda)\)
supported on the closed interval \([\lambda_{\min},\lambda_{\max}]\).}

The first arXiv version of H.Stahl's Theorem appeared in \cite{S1}, the latest arXiv version -
in \cite{S2}, the journal publication - in \cite{S3}.
The proof of Herbert Stahl is based on ingenious considerations related to
Riemann surfaces of algebraic functions. In \cite{E1},\cite{E2} a simplified version of the Herbert Stahl proof is presented. The proof, presented in \cite{E1},\cite{E2}, preserves all the main ideas of Stahl; the simplification
consists in technical details.

In the present paper we present a proof the BMV conjecture for the special case
\(\textup{rank}\,A=1\). Our proof is based on an elementary argument which
does not require complex analysis.

\section{Exponentially convex functions.}

 \begin{defn}
 \label{decf}
A function \(f\) on \(\mathbb{R}\), \(f:\,\mathbb{R}\to[0,\infty)\), is said to be \emph{exponentially convex}
if
\begin{enumerate}
\item[\textup{1}.]
 For every nonnegative integer \(N\), for every choice of real numbers \(t_1\), \(t_2\),\(\,\ldots\,\), \(t_{N}\), and complex numbers
\(\xi_1\), \(\xi_2, \,\ldots\,, \xi_{N}\), the inequality holds
\begin{equation}
\label{pqf}
\sum\limits_{r,s=1}^{N}f(t_r+t_s)\xi_r\overline{\xi_s}\geq 0;
\end{equation}
\item[\textup{2}.]
The function \(f\) is continuous on \(\mathbb{R}\).
\end{enumerate}
\end{defn}
The class of exponentially convex functions was introduced by S.N.Bernstein, \cite{Ber1}, see \S 15 there. Russian translation of the paper \cite{Ber1} can be found in
\cite[pp.\,370–425]{Ber2}.

From \eqref{pqf} it follows that
the inequality
\begin{math}
f(t_1+t_2)\leq\sqrt{f(2t_1)f(2t_2)}
\end{math}
holds for every \(t_1\in\mathbb{R},t_2\in\mathbb{R}\). Thus the alternative takes place: \\
\textit{If \(f\) is an exponentially convex function, then either \(f(t)\equiv 0\), or \(f(t)>0\) for every \(t\in\mathbb{R}\).}
\newpage
\noindent
\begin{center}
\textbf{Properties of the class of exponentially convex functions.}
\end{center}
\begin{enumerate}
\item[\textup{P\,1}.] If \(f(t)\) is an exponentially convex function and \(c\geq0\) is a nonnegative constant, then the function \(cf(t)\) is exponentially convex.
\item[\textup{P\,2}.] If \(f_1(t)\) and \(f_2(t)\) are exponentially convex functions, then their sum
\(f_1(t)+f_2(t)\) is exponentially convex.
\item[\textup{P\,3}.] If \(f_1(t)\) and \(f_2(t)\) are exponentially convex functions, then their product
\(f_1(t)\cdot f_2(t)\) is exponentially convex.
\item[\textup{P\,4}.] Let \(\lbrace f_{n}(t)\rbrace_{1\leq n<\infty}\) be a sequence of exponentially
convex functions. We assume that for each \(t\in\mathbb{R}\) there exists the limit
\(f(t)=\lim_{n\to\infty}f_{n}(t)\), and that \(f(t)<\infty\ \forall t\in\mathbb{R}\).
Then the limiting function \(f(t)\) is exponentially convex.
\end{enumerate}

From the functional equation for the exponential function
it follows that for each real number \(\mu\), for every choice of real numbers \(t_1,t_2,\,\ldots\,\), \(t_{N}\) and complex numbers
\(\xi_1\), \(\xi_2, \,\ldots\,, \xi_{N}\), the equality holds
\begin{equation}
\label{ece}
\sum\limits_{r,s=1}^{N}e^{(t_r+t_s)\mu}\xi_r\overline{\xi_s}=
\bigg|\sum\limits_{p=1}^{N}e^{t_p\mu}\xi_p\,\bigg|^{\,2}\geq 0.
\end{equation}
The relation \eqref{ece} can be formulated as
\begin{lem}
\label{ECE}
For each real number \(\mu\), the function \(e^{t\mu}\) of the variable \(t\) is exponentially convex.
\end{lem}

The following result is well known.
\begin{thm}[The representation theorem]\label{RepTe} {\ }\\[-2.0ex]
\begin{enumerate}
\item[\textup{1}.]
Let \(\sigma(d\mu)\) be a nonnegative measure on the real axis,
and let the function \(f(t)\) be defined as the two-sided Laplace transform of the measure
 \(\sigma(d\mu)\):
\begin{equation}
\label{rep}
f(t)=\int\limits_{\mu\in\mathbb{R}}e^{t\mu}\,\sigma(d\mu),
\end{equation}
where the integral in the right hand side of \eqref{rep} is finite for any \(t\in\mathbb{R}\). Then the function \(f\) is exponentially convex.
 \item[\textup{2}. ] Let \(f(t)\) be an exponentially convex function. Then this function \(f\) can be
  represented on \(\mathbb{R}\)  as a two-sided Laplace transform \eqref{rep} of a nonnegative measure \(\sigma(d\mu)\).  \textup{(}In particular, the integral in the right
     hand side of \eqref{rep} is finite for any \(t\in\mathbb{R}\).\textup{)} The representing measure \(\sigma(d\mu)\) is unique.
 \end{enumerate}
\end{thm}

The assertion 1 of the representation theorem is an evident consequence of Lemma~\,\ref{ECE}, of the properties P\,1, P\,2, P\,4, and of the definition of the integration operation.

 The proof of the assertion 2 can be found in \cite{A},\,Theorem 5.5.4, and in \cite{Wid},\,Theorem 21.

Thus the Herbert Stahl theorem can be reformulated as follows:\\
\textit{Let \(A\) and \(B\) be Hermitian \(n\times{}n\) matrices.
Let the function \(f_{A,B}(t)\) is defined by \eqref{TrF}
for \(t\in(-\infty,\infty)\). Then the function \(f_{A,B}(t)\), considered as a function of the variable \(t\), is exponentially convex.}

\section{A special case of the Herbert Stahl theorem}
\label{SCHT}
\begin{lem}
\label{PEx}
Let \(M\) be a Hermitian matrix. Assume that all \emph{off-diagonal} entries of the
matrix \(M\) are non-negative. Then \emph{all} entries of the matrix exponential
\(e^{M}\) are non-negative.
\end{lem}
\begin{proof}
Since the matrix \(M\) is Hermitian, its diagonal entries are real.
If a positive number \(\rho\) is large enough, all entries of the matrix
\(M_{\rho}\stackrel{\textup{\tiny def}}{=}M+{\rho}I\), where \(I\) is the identity
matrix, are non-negative. We choose and fix such \(\rho\). All entries of the
matrix \(e^{M_{\rho}}\) are nonnegative. Moreover \(e^M=e^{-\rho}{\cdot}\,e^{M_{\rho}}\).
\end{proof}
\begin{lem}
\label{LDP}
Let \(L\) and \(M\) be Hermitian matrices of the same size, say \(n\times n\).
 We assume that
\begin{enumerate}
\item The matrix \(L\) is diagonal;
\item All off-diagonal entries of the matrix \(M\) are non-negative.
\end{enumerate}
Then each entry of the matrix function \(e^{Lt+M}\) is an exponentially convex
function of the variable \(t\). In particular the function
\(f_{L,M}(t)=\textup{trace}\,\big[e^{Lt+M}\big]\) is exponentially convex.
\end{lem}
\begin{proof} We use the Lie product formula\footnote{See \cite[Theorem 2.10]{Ha}.}:
\begin{equation}
e^{X+Y}=\lim\limits_{p\to\infty}\big(e^{\frac{X}{p}}e^{\frac{Y}{p}}\big)^p,
\label{Lpf}
\end{equation}
where \(X\) and \(Y\) are arbitrary square matrices of the same size.
Taking \(X=Lt,\,Y=M\), we obtain
\begin{equation}
\label{sLpf}
e^{Lt+M}=\lim\limits_{p\to\infty}\big(e^{\frac{Lt}{p}}\cdot e^{\frac{M}{p}}\big)^p.
\end{equation}
According to Lemma \ref{PEx}, all entries of the matrix \(e^{M/p}\) are non-negative numbers.
Since the matrix \(L\) is Hermitian, its diagonal entries are real numbers.
Therefore, the matrix function \(e^{Lt/p}\) is of form
\[e^{Lt/p}=\textup{diag}\,\big(e^{l_1t/p},\,\ldots\,,e^{l_nt/p}\big),\]
where \(l_1,\,\ldots\,,l_n\) are real numbers. The exponentials \(e^{l_jt/p}\)
are exponentially convex functions of \(t\). Each entry of the matrix
\(e^{\frac{Lt}{p}}\cdot e^{\frac{M}{p}}\) is a linear combination of these exponentials with non-negative coefficients. According to the properties P1 and P2,
the entries of the matrix function \(e^{\frac{Lt}{p}}\cdot e^{\frac{M}{p}}\) are
exponentially convex functions. Each entry of the matrix function
\(\big(e^{\frac{Lt}{p}}\cdot e^{\frac{M}{p}}\big)^p\) is a sum of products
of some entries of the matrix function \(e^{\frac{Lt}{p}}\cdot e^{\frac{M}{p}}\).
According to the properties P2 and P3, the entries of the matrix function
\(\big(e^{\frac{Lt}{p}}\cdot e^{\frac{M}{p}}\big)^p\) are exponentially convex functions. From the limiting relation \eqref{sLpf} and from the property P4 it follows that all entries of the matrix \(e^{Lt+M}\) are exponentially convex functions. All the more, the function \(f_{L,M}(t)\), which is the sum of the diagonal entries, is exponentially convex.
\end{proof}
\begin{thm}
Let \(A\) and \(B\) be Hermitian matrices of size \(n\times n\).
Assume moreover that  the matrix  \(A\) is of rank one. Then the function
\(f_{A,B}(t)\), defined by \eqref{TrF}, is exponentially convex.
\end{thm}
\begin{proof} {\ }\\
\textbf{1}. If \(U\) is an unitary matrix, then
\(f_{A,B}(t)=f_{UAU^{\ast}\!,\,UBU^{\ast}}(t).\)
 Since \(\textup{rank}\,A=1\), we can choose the matrix \(U\) such that
 the matrix \(UAU^{\ast}\) are of the form
 \(UAU^{\ast}=\textup{diag}(\lambda_1,\,\ldots\,,\lambda_{n-1},\lambda_n)\),
 where \(\lambda_j=0,\,1\leq j\leq n-1\), \(\lambda_n\not=0\).
 We choose and fix such unitary matrix \(U\).\\
 \textbf{2}. Thus from the very beginning we can assume that the matrices \(A\)
 and \(B\) are of the form
 \begin{equation*}
 A=
 \begin{bmatrix}
 \textup{\Large 0}_{n-1}&0_{n-1}\\
 0_{n-1}^{\ast}&\lambda_n
 \end{bmatrix}
 \ , \quad
 B=
 \begin{bmatrix}
 B_{n-1}&b_{n-1}\\
 b_{n-1}^{\ast}&\mu_n
 \end{bmatrix}
 \,,
 \end{equation*}
 where \(\textup{\Large 0}_{n-1}\) is the \emph{zero} matrix of size \((n-1)\times(n-1)\),
 \(0_{n-1}\) is the \emph{zero} column of size \(n-1\), \(\lambda_n\not=0\)
 is a real number, \(B_{n-1}\) is the a Hermitian matrix of size \((n-1)\times(n-1)\), \(b_{n-1}\) is a column of size \(n-1\),
 \(\mu_n\) is a real number.

 There exists an unitary matrix \(V_{n-1}\) of size \((n-1)\times(n-1)\)
 such that the matrix \(V_{n-1}B_{n-1}V_{n-1}^{\ast}\) is diagonal.
 We choose and fix such matrix \(V_{n-1}\). We define the matrices
 \begin{gather*}
 M_{n-1}\stackrel{\textup{\tiny def}}{=}V_{n-1}B_{n-1}V_{n-1}^{\ast},\quad g\stackrel{\textup{\tiny def}}{=}V_{n-1}b_{n-1}.
 \end{gather*}
 The matrix \(M_{n-1}\) is a diagonal matrix of size \((n-1)\times(n-1)\):
 \begin{equation*}
 M_{n-1}=\textup{diag}(\mu_1,\,\ldots\,,\mu_{n-1}).
 \end{equation*}
 The matrix \(g\) is a column of size \(n-1\):
 \begin{equation*}
 g=
 \begin{bmatrix}
 \gamma_1 \\
 \vdots \\
 \gamma_{n-1}
 \end{bmatrix},
 \end{equation*}
 where \(\gamma_j,\,1\leq j\leq n-1,\) are complex numbers. Let us define
 numbers \(\omega_j\) which satisfy the conditions
 \begin{gather*}
 |\omega_j|=1,\ \ \omega_j\cdot \gamma_j=|\gamma_j|,\qquad 1\leq j\leq n-1.
 \end{gather*}
 If \(\gamma_j\not=0\) for some \(j\), then such \(\omega_j\) is unique. If
 \(\gamma_j=0\) for some \(j\), we take \(\omega_j=1\). Let us define
 the diagonal matrix \(\Omega_{n-1}\) of size \((n-1)\times(n-1)\) as
 \begin{gather*}
 \Omega_{n-1}=\textup{diag}\,(\omega_1,\,\ldots,,\omega_{n-1}).
 \end{gather*}
According to the construction of the matrix \(\Omega_{n-1}\),
all entries of the column
\begin{equation*}
|g|\stackrel{\textup{\tiny def}}{=}\Omega_{n-1}\,g
\end{equation*}
are non-negative numbers:
\begin{equation*}
 |g|=
 \begin{bmatrix}
 |\gamma_1| \\
 \vdots \\
 |\gamma_{n-1}|
 \end{bmatrix}.
 \end{equation*}
 Moreover
 \begin{equation*}
 \Omega_{n-1}M_{n-1}\Omega_{n-1}^{\ast}=M_{n-1}.
 \end{equation*}
 We introduce the matrices \(W_{n-1}=\Omega_{n-1}V_{n-1}\),
 \begin{equation*}
 W=
 \begin{bmatrix}
 W_{n-1}&0_{n-1}\\
 0_{n-1}^{\ast}&1
 \end{bmatrix}\,\cdot
 \end{equation*}
  The equalities
\begin{equation*}
WAW^{\ast}=L,\qquad WBW^{\ast}=M
\end{equation*}
hold, where
 \begin{equation*}
 L=
 \begin{bmatrix}
 \textup{\Large 0}_{n-1}&0_{n-1}\\
 0_{n-1}^{\ast}&\lambda
 \end{bmatrix}
 \ , \quad
 M=
 \begin{bmatrix}
 M_{n-1}&|g|\\
 |g|^{\ast}&\mu_n
 \end{bmatrix}
 \,\cdot
\end{equation*}
\textbf{3}.
The matrix \(W\) is unitary. Therefore
\begin{equation*}
f_{A,B}(t)=f_{L,M}(t).
\end{equation*}
The matrix \(L\) is diagonal. (Actually \(L=A\) .) Off-diagonal entries
of the matrix \(M\) are non-negative numbers:
\begin{equation*}
m_{j,k}=0,\  1\leq j<k<n;\qquad m_{j,n}=|\gamma_j|, \ 1\leq j\leq n-1.
\end{equation*}
According to Lemma \ref{LDP}, the function \(f_{L,M}(t)\) is exponentially convex.
\end{proof}

\end{document}